\documentclass[reqno, 11pt]{amsart}
\usepackage{amssymb,latexsym, bbm,amsmath, amsthm, mathrsfs, xy,color}
\usepackage{tikz-cd}

\usepackage[shortlabels]{enumitem}

\usepackage{setspace}

\xyoption{all}

\newtheorem{mythm}{Theorem}[section]

\newtheorem{myprop}[mythm]{Proposition}
\newtheorem{mycor}[mythm]{Corollary}

\theoremstyle{definition}
\newtheorem{mydef}[mythm]{Definition}

\newtheorem{myassu}[mythm]{Assumption}

\theoremstyle{remark}
\newtheorem{myrem}[mythm]{Remark}

\DeclareMathOperator{\im}{Im}
\DeclareMathOperator{\Ext}{Ext}
\DeclareMathOperator{\Gal}{Gal}
\DeclareMathOperator{\Hom}{Hom}

\DeclareMathOperator{\tr}{tr}
\DeclareMathOperator{\Tot}{Tot}
\DeclareMathOperator{\Frob}{Frob}
\DeclareMathOperator{\Symm}{Symm}
\DeclareMathOperator{\ord}{ord}
\DeclareMathOperator{\cond}{cond}

\newcommand{\C}{\mathbf{C}}
\newcommand{\Q}{\mathbf{Q}}
\newcommand{\Z}{\mathbf{Z}}
\newcommand{\F}{\mathbf{F}}

\newcommand{\T}{\mathbf{T}}
\newcommand{\R}{\mathbf{R}}

\newcommand{\Fil}{\text{Fil}}

\numberwithin{equation}{section}

\title[Galois deformation rings and modularity]{Galois deformation rings and modularity in the residually reducible case}
\author[Geoffrey Akers]{Geoffrey Akers}

\begin{document}

\begin{abstract}
We prove some residually reducible $p$-adic Galois representations of number fields arise from modular forms.  We study the universal deformation ring arising from deformations satisfying the Fontaine--Laffaille condition at primes over $p$.  Under certain conditions, we establish the reduced universal deformation ring is a discrete valuation ring.  The method uses pseudocharacters and certain bounds on Selmer groups, where the ideal of reducibility as defined by Bella{\"i}che and Chenevier is shown to be maximal and principal.  A self-dual assumption is not needed in our argument.  The main result on deformations applies to $n$-dimensional representations.  In applications, a congruence between Hermitian modular forms due to Klosin is used with our results to prove modularity of some 4-dimensional $p$-adic representations of the imaginary quadratic field $\Q(i)$.
\end{abstract}



\thanks{Electronic version of an article published as
\emph{International Journal of Number Theory}, Vol.~21, No.~2 (2025), pp.~449--471,
DOI: 10.1142/S1793042125500228,
\copyright\ World Scientific Publishing Company,
\texttt{https://www.worldscientific.com/worldscinet/ijnt}.}

\maketitle
\setcounter{tocdepth}{3}

\section{Introduction}

Let $F$ be a number field and let $\rho:\Gal(\overline F / F)\to GL_n(\overline{\Q_p})$ be an irreducible Galois representation unramified outside a finite set of places of $F$.  Suppose $\rho$ is valued in $GL_n(E)$ with associated residual representation $\overline \rho$, where $E$ is a finite extension of $\Q_p$ with ring of integers $\mathcal{O}$ and uniformizer $\varpi$.  The modularity of such $\rho$ with additional hypotheses on $\rho$ and $\overline \rho$, including on the restriction of $\rho$ to the decomposition groups $D_v$ for each place $v$ of $F$ lying over $p$, has been proven in many cases by establishing that certain Hecke algebras are universal deformation rings parametrizing deformations of certain mod $p$ representations, known as $R=\T$ theorems.  In this paper we treat residually reducible $\rho$, and assume the associated residual representation $\overline \rho$ has semi-simplification $\overline \rho ^{ss} \cong \tau_1 \oplus \tau_2$ with $\tau_1$ and $\tau_2$ non-isomorphic and absolutely irreducible representations.  The strategy is to consider a universal deformation ring for a non-split residual representation $\tau$ of $G_\Sigma = \Gal (F_\Sigma /F)$, the Galois group of the maximal extension of $F$ unramified outside $\Sigma$, with $\tau^{ss}\cong \tau_1 \oplus \tau_2$.

In dimension $n=2$, a deformation $\rho':G_\Sigma \to GL_2(A)$ is ordinary if it satisfies
$$
\rho'|_{D_v} \cong \begin{bmatrix}
\eta_1 & * \\
0 & \eta_2
\end{bmatrix}
$$ for $v|p$ where $\eta_1$ is an unramified character.  Modularity results and $R=\T$ theorems were proved for certain ordinary representations for $F=\Q$ by Skinner and Wiles \cite{skinner1997ordinary}, and for ordinary representations satisfying an additional minimal ordinary condition with $F$ imaginary quadratic by Berger and Klosin \cite{berger2009deformation, berger2011r}.  Some $R=\T$ theorems for non-ordinary representations were proved in \cite{calegari2006eisenstein}.  In dimension $n\geq 2$ and $F$ a number field, $R=\T$ theorems were proved for crystalline representations by Berger and Klosin \cite{berger2013deformation}.  The modularity results of \cite{skinner1999residually} in the ordinary case for $F=\Q$ and $F$ totally real with $n=2$ were proved by a different method since there is not a suitable surjective map $R\to \T$.  For an ordinary pseudo-deformation ring $R$ and a Hida Hecke algebra $\T$, Wake and Wang-Erickson \cite{wake2017ordinary} proved $R=\T$ and established some of the results of \cite{skinner1999residually} by a new technique.

In \cite{skinner1997ordinary}, $\tau^{ss}\cong \chi \oplus 1$ and conditions on $\Sigma$, $\chi$, and $p$ ensure $\tau$ is the unique non-split extension of 1 by $\chi$.  The existence of a certain minimal ordinary reducible characteristic zero deformation leads to an associated Galois cohomology group that is compared with a Hecke congruence module.  An $R=\T$ theorem for ordinary minimal deformations follows from a numerical criterion of Lenstra and Wiles, and the authors then prove $R=\T$ in the ordinary case.  In \cite{calegari2006eisenstein} for some non-ordinary deformations and in \cite{berger2009deformation} for minimal ordinary deformations of $\tau$, a non-trivial extension of $\chi$ by 1 for a certain character $\chi$, the authors prove $R=\T$ theorems by establishing $R$ is a discrete valuation ring.  Also using the method of \cite{bellaiche2006lissite} with $\rho^{R}$ the universal deformation corresponding to $R$, these authors study the ideal of reducibility defined to be the smallest ideal $I$ such that $\tr(\rho^{R}) \mod I$ is the sum of two characters.  A crucial ingredient in their method is the assumption that the modules of extensions of $\tau_2$ by $\tau_1$ and $\tau_1$ by $\tau_2$ have $\F$-dimension 1.  The method to prove $R$ is DVR in \cite{calegari2006eisenstein} and \cite{berger2009deformation} also requires there to be no non-trivial upper-triangular infinitesimal deformations (to $\F [x]/x^2$).  In these works, the reducibility ideal is maximal if and only if there does not exist a surjective map $R/I \to \F[x]/x^2$ or $R/I \to \mathcal{O}/\varpi^2 \mathcal{O}$.  By universality of $R$, these surjective maps correspond to suitable non-trivial upper-triangular deformations to $GL_2(\F[x]/x^2)$ or $GL_2 (\mathcal{O}/\varpi^2 \mathcal{O})$.

In \cite{berger2011r}, with $n=2$, and \cite{berger2013deformation}, where $n\geq 2$ and $\tau^{ss}\cong \tau_1\oplus \tau_2$, the authors establish commutative algebra criteria for establishing $R=\T$ for $R$ not necessarily a DVR.  In these works, it remains true that there are no non-trivial (block) upper-triangular infinitesimal deformations (of type ordinary minimal in \cite{berger2011r} and crystalline in \cite{berger2013deformation}).  In dimension $n=2$, the ideal of reducibility is principal following \cite{bellaiche2006lissite} and \cite{calegari2006eisenstein} and $R=\T$ is established provided $R/I$ is a cyclic Artinian module with the size of $R/I$ providing a lower bound on the size of $\T/J$ for some Eisenstein ideal $J$.  In applying the commutative algebra criteria \cite[Proposition 6.9]{berger2011r} and \cite[Theorem 4.1]{berger2013deformation}, the principality of $I$ is used to prove a surjective map $R\to \T$ is an isomorphism.  In \cite{berger2013deformation}, where $n\geq 2$, the principality of the reducibility ideal is established under a self-dual hypothesis.  In the method of Berger--Klosin, the lack of non-trivial infinitesimal upper-triangular deformations is used to show $R/I = \mathcal{O}/\varpi^s$ for some $s>0$, regardless of $I$ being maximal in $R$.  An upper-bound for $s$ is determined by a bound on a Selmer group (see Theorem 8.5 and Proposition 7.14 in \cite{berger2013deformation}).

There are differences between the methods of the current article with those of \cite{skinner1997ordinary, calegari2006eisenstein, berger2009deformation, berger2011r, berger2013deformation}.  Our setup of Fontaine--Laffaille deformation problems is similar to \cite{berger2013deformation}, but we work with the reduced universal deformation ring.  In \cite{berger2013deformation}, two methods for the principality of the reducibility ideal are given, but in this article we use a different method.  We establish sufficient conditions for $R$ to be DVR.  In contrast to the "essentially self-dual" deformations studied in \cite[Theorem 2.11]{berger2013deformation} and \cite[Sec. 1.8]{bellaiche2009families}, we make no assumption on an anti-automorphism of $R[G]$ in Section \ref{sec:reducibility}.  We use that there is only one non-split extension of the two representations $\tau_1$ and $\tau_2$ (in either order) in place of a self-dual assumption to prove the principality of $I$.  The principality of $I$ in this setup is a main result of the current paper.  Under our assumptions, the main ones consisting of bounds on Selmer groups, the principality of the reducibility ideal is established by showing the injective map used to construct explicit extensions of $\tau_1$ by $\tau_2$ has image in a particular Selmer group.  As in \cite{berger2013deformation}, our $\tau$ has no non-trivial upper-triangular infinitesimal deformations.  Our assumptions on Selmer groups ensure that there is a unique extension of $\tau_2$ by $\tau_1$ and that there is no non-trivial Fontaine--Laffaille upper-triangular deformation of $\tau$ to $\mathcal{O}/\varpi^2 \mathcal{O}$.  So $I$ is also maximal in $R$.

As an application of the main result we prove some modularity results for irreducible, residually reducible representations.  We make use of congruences between modular forms constructed by others.  This method has origins in the work of Ribet using congruences between cusp forms and Eisenstein series in proving the converse to Herbrand's theorem \cite{ribet1976modular}.  The irreducible Galois representation $\rho_g$ attached to a cusp form $g$ satisfying such a congruence will then have an associated mod $p$ representation that is reducible, not semi-simple, and has the desired semi-simplification.  Congruences between automorphic forms on unitary groups have been the subject of considerable interest \cite{skinner2014iwasawa, brown2020congruence, klosin2009congruences, klosin2015maass}.  A congruence between a Maass lift of an ordinary at $p$ eigenform and a non-CAP Hermitian modular form is analogous to a congruence between an Eisenstein series and a cusp form.  Let $K$ be an imaginary quadratic field.  In this paper, we consider the unitary group $U(2,2)$ associated to $K=\Q(i)$ over $\Q$.  We establish modularity results for 4-dimensional representations of $\Gal(K_\Sigma/K)$ using congruences constructed by Klosin \cite{klosin2009congruences} under certain conditions including on special $L$-values (see Section \ref{sec_congruence}).  We remark that relevant congruences exist when $K$ has odd discriminant \cite{klosin2015maass} and our theorem applies to those fields.

We assume $n! \in R^\times$ in Section \ref{sec:reducibility}, where $n$ is the dimension of the representations, and in the results of Section \ref{sec:r_dvr} we assume $p\nmid n!$, which means we exclude the prime $p=2$.  This condition implies $n!\in \F^\times$, and appears in the definition of pseudocharacter (cf. \cite[Sec. 1.2]{bellaiche2009families}).

We now give an outline of this paper.  In Section \ref{setup}, we describe the Fontaine--Laffaille representations and Selmer groups used throughout.  In Section \ref{sec:reducibility}, we study the ideal of reducibility and work with a general group $G$ (there is no mention of Galois groups).  We study certain quotients of $R[G]$ using tools from \cite{bellaiche2009families}, and find the reducibility ideal $I$ is an ideal of $R$.  We give sufficient conditions for $I$ to be principal.  In Section \ref{sec:r_dvr}, we prove the main result on Galois representations (Theorem \ref{main_thm}).  The principality of $I$ is established by showing the injective map used to construct explicit extensions has image in a particular Selmer group (see Proposition \ref{prop_selmer}).  In Section \ref{sec:2dim} for $2$-dimensional representations of $\Q$ and Section \ref{sec:4dim} for $4$-dimensional representations of $\Q(i)$ we apply our result to modularity problems.

\section*{Notations}
For any field field $L$, we let $\overline L$ denote a fixed algbraic closure and we write $G_L = \Gal(\overline L / L)$ for the absolute Galois group of $L$.  For a number field $F$, we denote by $Cl_F$ the class group of $F$, and for a place $v$ of $F$ we will denote by $F_v$ the completion of $F$ at $v$.

\section{Setup}\label{setup}
Let $F$ be a number field and $p>2$ a prime satisfying $p\nmid \# Cl_F$ with $p$ unramified in $F/\Q$.  Let $\Sigma$ be a finite set of finite places of $F$ containing all the places of $F$ lying over $p$.  Let $G_\Sigma$ denote the Galois group $\Gal(F_\Sigma / F)$, where $F_\Sigma$ is the maximal extension of $F$ unramified outside $\Sigma\cup \left\{ v|\infty \right\}$.  For each prime $\mathfrak{q}$ of $F$, we fix compatible embeddings $\overline F \hookrightarrow \overline {F_\mathfrak{q}} \hookrightarrow \C$ and write $D_\mathfrak{q} \subset G_F$ for a decomposition subgroup of $G_F$.  By a slight abuse of notation, also write $D_\mathfrak{q}$ for the image of $G_{F_\mathfrak{q}}$ in $G_\Sigma$.  Let $E$ be a (sufficiently large) finite extension of $\Q_p$ with ring of integers $\mathcal{O}$, uniformizer $\varpi$, and residue field $\F$.

\subsection{Fontaine--Laffaille representations}\label{f-l}
We recall some of the theory of Fontaine and Laffaille \cite{fontaine1982construction}, following the exposition in \cite[Sec. 5.2.1]{berger2013deformation}.  Let $MF_\mathcal{O}$ denote the category of filtered Dieudonn\'e modules $M$ over $\mathcal{O}$ together with a decreasing filtration $\Fil^i M$ by $\mathcal{O}$-submodules which are $\mathcal{O}$-direct summands with $\Fil^0 M = M$ and $\Fil^{p-1}M = \left\{ 0 \right\}$ and Frobenius linear maps $\Phi^i: \Fil^i M \to M$ with $\Phi^i|_{\Fil^{i+1}M}= p\Phi^{i+1}$ and $\sum_i \Phi^i \Fil^i M = M$.  The definition will use the (covariant) Fontaine--Laffaille functor $\mathbf{G}$ from $MF_\mathcal{O}$ to the category of finitely generated $\mathcal{O}$-modules with a continuous action of $G_K$, where $K$ is an unramified extension of $\Q_p$ (see \cite[Sec. 2.4.1]{clozel2008automorphy} for the definition of $\mathbf{G}$).

We recall the definition of a \textit{Fontaine--Laffaille representation} following Section 2.4.1 of \cite{clozel2008automorphy}:  Let $v$ be a place of $F$ with $v|p$ and let $A$ be a complete Noetherian $\Z_p$-algebra with residue field $\F$.  A continuous representation $\rho: D_v \to GL_n(A)$ is Fontaine--Laffaille if for each Artinian quotient $A'$ of $A$, $\rho\otimes A'$ lies in the essential image of the Fontaine--Laffaille functor $\textbf{G}$.  We also call a continuous finite-dimensional $G_\Sigma$-representation $V$ over $\Q_p$ Fontaine--Laffaille if, for all $v|p$, it is crystalline and $\Fil^0 D = D$ and $\Fil^{p-1} D = (0)$ for the filtered vector space $D=(B_{crys}\otimes_{\Q_p} V)^{D_v}$ defined by Fontaine.

\subsection{Selmer groups}
If $L$ over $K$ is a Galois extension and $M$ is a topological $\Gal(L/K)$-module, we denote by $H^1(\Gal(L/K),M)$ the continuous cohomology group.  If $L$ is a separable algebraic closure of $K$ then we write
$H^1(K,M)$ for $H^1(G_K, M)$.

Let $K$ denote an unramified extension of $\Q_p$.  For $V$ a finite-dimensional $\Q_p$-vector space with a continuous $G_K$ action, we put, following \cite[3.7.2]{bloch1990functions},
$$
H^1_f (K,V) = \ker\left(H^1 (K, V) \to H^1(K,B_{crys}\otimes V) \right).
$$
Let $T$ be a $G_K$-stable $\mathcal{O}$-lattice in an $E$-vector space $V$.  With the natural maps
$$
H^1(F_v, T) \to H^1(F_v, V) \to H^1(F_v, V/T),
$$
we set $H^1_f (F_v, T) \subset H^1 (F_v, T)$ to be the inverse image of $H^1_f(F_v, V)$, and we set $H^1_f (F_v, V/T) \subset H^1 (F_v, V/T)$ to be the image of $H^1_f(F_v, V)$.

For a $p$-adic $G_\Sigma$-module $M$ (finitely generated or finite corank over $\mathcal{O}$; see Section 5 of \cite{berger2013deformation} for details), we define the global Selmer group $H^1_\Sigma (F,M)$ as a subgroup of $H^1(G_\Sigma, M)$ consisting of cohomology classes that are crystalline at all $v|p$:
$$
H^1_\Sigma (F,M) = \ker \left( H^1 (G_\Sigma, M) \to \prod_{v|p}H^1(F_v,M)/H^1_f(F_v, M)   \right) .
$$

The Fontaine--Laffaille theory is used to define the local condition when $M$ is a $p$-adic $G_K$-module of finite cardinality:  If $K$ denotes an unramified extension of $\Q_p$, and if $M$ is in the essential image of \textbf{G}, we define $H^1_f(K, M)$ as the image of $\Ext^1_{MF_\mathcal{O}} (1_{FD}, D)$ in $H^1(K,M) \cong \Ext^1_{\mathcal{O}[G_K]}(1,M)$ where $\mathbf{G} (D) = M$ and $1_{FD}$ is the unit filtered Dieudonn\'e module defined in Lemma 4.4 of \cite{bloch1990functions} (also see Section 5 of \cite{berger2013deformation}).

\section{Reducibility Ideal}\label{sec:reducibility}
Let $R$ be a local Henselian Noetherian commutative ring which is reduced with maximal ideal $m_R$ and residue field $\F$.  Let $G$ be a group and let $\rho : R[G] \to M_n(R)$ be a morphism of $R$-algebras.  Set $T=\tr \rho : R[G] \to R$.  We assume that $n!$ is invertible in $R$ and assume
$$
\rho^{} = \begin{bmatrix}
\tau_1 & * \\
0 & \tau_2
\end{bmatrix} \quad \mod m_R
$$
is a non-split extension of $\tau_2$ by $\tau_1$ for two non-isomorphic absolutely irreducible representations $\tau_i$  over $\F$ of dimension $n_i$, $i=1,2$.

\begin{mydef}[Section 1.2.4 in \cite{bellaiche2009families}]
Let $S$ be an $R$-algebra and let $T:S\to R$ be a pseudocharacter (for a definition of pseudocharacter see Section 1.2.1 of \cite{bellaiche2009families}).  The \textit{kernel} of $T$ is the two-sided ideal of $S$ defined by
$$
\ker T  = \left\{x \in S : \forall y \in S, T(xy)=0 \right\}.
$$
When $S \to Q$ is a surjective morphism of $R$-algebras whose kernel is included in $\ker T$, then $T$ factors uniquely as a pseudocharacter $T_Q : Q \to R$, which will often be denoted by $T$ (as in \cite[Sec. 1.2.4]{bellaiche2009families}).
\end{mydef}

Let $S^T$ be $R[G]/\ker T$ and let $S^\rho$ be $R[G]/\ker \rho$.  Both $S^T$ and $S^\rho$ are Cayley-Hamilton quotients of $R[G]$ (see \cite{bellaiche2009families} Sec. 1.2.5).  There is a canonical $R$-algebra map $\varphi:S^\rho \to S^T$.  Write $\Tot(R)$ for the total fraction ring of $R$.

\begin{mydef}[Definition 1.3.1 in \cite{bellaiche2009families}]\label{def_gma}
Let $S$ be an $R$-algebra.  Then $S$ is a \textit{Generalized Matrix Algebra (GMA)} of type $(n_1, n_2)$ if $S$ is equipped with:
\begin{enumerate}
\item a pair of orthogonal idempotents $e_1, e_2$ of sum 1,
\item for each $i$, an $R$-algebra isomorphism $\psi_i: e_i S e_i \to M_{n_i}(R)$
\end{enumerate}
such that the trace map $T:S\to R$, defined $T(x)=\sum^2_{i=1}\tr(\psi_i (e_i xe_i))$, satisfies $T(xy)=T(yx)$ for all $x,y \in S$.  We call $\mathcal{E} = \left\{ e_i, \psi_i; i = 1,2\right\}$ the \textit{data of idempotents} of $S$.
\end{mydef}

\begin{mydef}[Definition 1.3.6 in \cite{bellaiche2009families}]\label{def_adapted}
Let $S$ be a GMA of type $(n_1, n_2)$ with data of idempotents $\mathcal{E} = \left\{ e_i, \psi_i; i = 1,2\right\}$.  Let $B$ be a commutative $R$-algebra.  Then a morphism of $R$-algebras $\rho: S \to M_n(B)$ is called \textit{adapted to $\mathcal{E}$} if its restriction $\rho|_{e_1 S e_1 \oplus e_2 S e_2}$ equals the composite of $\psi_1 \oplus \psi_2$ with the diagonal map $M_{n_1}(B) \oplus M_{n_2}(B) \to M_n(B)$ where $n=n_1 + n_2$.
\end{mydef}

\begin{mythm}\label{thm_data}
\hspace{2em}
\begin{enumerate}
\item There is a data of idempotents $\mathcal{E}^T = \left\{ e^T_i, \psi^T_i; i = 1,2\right\}$ on $S^T$ for which $S^T$ is a GMA.  The data $\mathcal{E}^T$ defines $R$-submodules $\mathcal{A}^T_{i,j}$ of $S^T$ that satisfy
$$
\mathcal{A}^T_{i,j} \mathcal{A}^T_{j,k}\subset \mathcal{A}^T_{i,k}, \quad T:\mathcal{A}^T_{i,i}\overset{\sim}\to R, \quad i,j,k\in \left\{1,2\right\}.
$$
For $i\neq j$,
$$
T(\mathcal{A}^T_{i,j} \mathcal{A}^T_{j,i})\subset m_R.
$$
One has
$$
S^T \cong
\begin{bmatrix} 
M_{n_1}(\mathcal{A}^T_{1,1}) & M_{n_1,n_2}(\mathcal{A}^T_{1,2})\\
M_{n_2,n_1}(\mathcal{A}^T_{2,1}) & M_{n_2}(\mathcal{A}^T_{2,2})\\
\end{bmatrix}
$$
is an isomorphism of $R$-algebras.

\item For $\mathcal{E}^T$ as in (1),
$$
S^T \cong
\begin{bmatrix} 
M_{n_1}(A^T_{1,1}) & M_{n_1,n_2}(A^T_{1,2})\\
M_{n_2,n_1}(A^T_{2,1}) & M_{n_2}(A^T_{2,2})\\
\end{bmatrix}
\subset M_n(\Tot(R))
$$
is an isomorphism of $R$-algebras where $A^T_{i,j}$ are fractional ideals of $R$ that satisfy
$$
A^T_{i,j} A^T_{j,k}\subset A^T_{i,k}, \quad A^T_{i,i} = R, \quad i,j,k\in \left\{1,2\right\}.
$$
For $i\neq j$,
$$
A^T_{i,j} A^T_{j,i} \subset m_R .
$$

\item There is a data of idempotents $\mathcal{E}^\rho=\left\{ e^\rho_i, \psi^\rho_i; i = 1,2\right\}$ on $S^\rho$ for which $S^\rho$ is a GMA.  The data $\mathcal{E}^\rho$ defines $R$-submodules $\mathcal{A}^\rho_{i,j}$ of $S^\rho$ that satisfy
$$
\mathcal{A}^\rho_{i,j} \mathcal{A}^\rho_{j,k}\subset \mathcal{A}^\rho_{i,k}, \quad T:\mathcal{A}^\rho_{i,i}\overset{\sim}\to R, \quad i,j,k\in \left\{1,2\right\}.
$$
For $i\neq j$,
$$
T(\mathcal{A}^\rho_{i,j} \mathcal{A}^\rho_{j,i})\subset m_R.
$$
One has
$$
S^\rho \cong
\begin{bmatrix} 
M_{n_1}(\mathcal{A}^\rho_{1,1}) & M_{n_1,n_2}(\mathcal{A}^\rho_{1,2})\\
M_{n_2,n_1}(\mathcal{A}^\rho_{2,1}) & M_{n_2}(\mathcal{A}^\rho_{2,2})\\
\end{bmatrix}
$$
is an isomorphism of $R$-algebras.

\item For $\mathcal{E}^\rho$ as in (3),
$$
S^\rho \cong
\begin{bmatrix} 
M_{n_1}(A^\rho_{1,1}) & M_{n_1,n_2}(A^\rho_{1,2})\\
M_{n_2,n_1}(A^\rho_{2,1}) & M_{n_2}(A^\rho_{2,2})\\
\end{bmatrix}
\subset M_n(R)
$$
is an isomorphism of $R$-algebras where $A^\rho_{i,j}$ are ideals of $R$ that satisfy
$$
A^\rho_{i,j} A^\rho_{j,k}\subset A^\rho_{i,k}, \quad A^\rho_{i,i}=R, \quad i,j,k\in \left\{1,2\right\}.
$$
For $i\neq j$,
$$
A^\rho_{i,j} A^\rho_{j,i} \subset m_R.
$$

\item The data $\mathcal{E}^T$ and $\mathcal{E}^\rho$ can be chosen so that $\varphi:S^\rho \to S^T$ satisfies
$$
\varphi(\mathcal{A}_{i,j}^\rho)=\mathcal{A}_{i,j}^T.
$$

\end{enumerate}
\proof
For (1) and (3) this is Theorem 1.4.4(i) of \cite{bellaiche2009families}, and item (2) follows from Theorem 1.4.4(ii) \cite{bellaiche2009families} because for us $R$ is reduced.

For (4), one can conjugate $\rho$ by an invertible matrix $P$ with values in $R$ to get $\rho$ adapted to $\mathcal{E}^\rho$ from (3) by \cite{bellaiche2009families} Lemma 1.3.7 (we can take $P\in GL_n(R)$ because every finite type projective $R$-module is free).  By \cite{bellaiche2009families} Proposition 1.3.8
$$
S^\rho \cong \rho(R[G])=
\begin{bmatrix} 
M_{n_1}(A^\rho_{1,1}) & M_{n_1,n_2}(A^\rho_{1,2})\\
M_{n_2,n_1}(A^\rho_{2,1}) & M_{n_2}(A^\rho_{2,2})\\
\end{bmatrix}
$$
is a GMA associated to the ideals $A^\rho_{i,j}$.

Item (5) is done as in Lemma 2.5 of \cite{berger2013deformation}.
\qed
\end{mythm}

\begin{mydef}[Definition 1.5.2 in \cite{bellaiche2009families}]
The \textit{ideal of reducibility} of $T$ is the ideal $I$ of $R$ such that for all proper ideals $J$ of $R$, one has $I\subseteq J$ if and only if there exist pseudocharacters $T_1,T_2: R[G]\otimes R/J \to R/J$ such that
\begin{enumerate}
\item $T\otimes R/J = T_1 + T_2$, and
\item $T_i\otimes \F = \tr \: \tau_i$, $\:$ for $i=1,2$.
\end{enumerate}
\end{mydef}

\begin{myprop}[Proposition 1.5.1 in \cite{bellaiche2009families}]\label{prop_image_ideal}
The ideal of reducibility $I$ exists.  For $? \in \left\{ \rho, T\right\}$ and $S^?$ with data of idempotents $\mathcal{E}^?$ as in Theorem \ref{thm_data},
$$
I = T(\mathcal{A}^?_{1,2}\mathcal{A}^?_{2,1}).
$$
\proof
The theorem and proof are \cite{bellaiche2009families} Proposition 1.5.1 with the partition $\mathcal{P} = \left\{ \left\{1 \right\}, \left\{2 \right\} \right\}$.  We only note that the given formula for $I$ does not depend on choice of Cayley-Hamilton quotient of $R[G]$ such as $S^\rho$ because $\varphi$ satisfies $T\circ \varphi = T$.  We mean the pseudocharacter $T:R[G]\to R$ induces $T_\rho: S^\rho \to R$ and $T_T:S^T\to R$.  Then $T_T \circ \varphi = T_\rho$ holds and
\begin{align*}
T_\rho (\mathcal{A}_{1,2}^\rho \mathcal{A}_{2,1}^\rho) &= T_T(\varphi (\mathcal{A}_{1,2}^\rho \mathcal{A}_{2,1}^\rho)) \\ 
&= T_T (\mathcal{A}_{1,2}^T \mathcal{A}_{2,1}^T).
\end{align*}
\qed
\end{myprop}

\begin{myprop}
Let $S^\rho$ and $\mathcal{E}^\rho$ be as in Theorem \ref{thm_data}.  Then
$$
I={A}^\rho_{1,2}{A}^\rho_{2,1}.
$$
\proof
As in the proof of Theorem \ref{thm_data}(4), $\rho:R[G]\to M_n(R)$ factors through $S^\rho$ and since $R$ is local (and so every finite type projective $R$-module is free) there is $P \in GL_n(R)$ such that $P\rho P^{-1}:S^\rho \to M_n(R)$ is adapted to $\mathcal{E}^\rho$ by Lemma 1.3.7 of \cite{bellaiche2009families}.  Let $E_i \in e_i^\rho S^\rho e_i^\rho$ for $i=1,2$ be the unique element such that $\psi_i^\rho (E_i)$ is the elementary matrix of $M_{n_i}(R)$ with entry 1 at row 1 and column 1.  By definition of the $E_i$'s and adaptedness, $P\rho P^{-1}(E_i S^\rho E_j)$ is the $R$-module of matrices with coefficients 0 everywhere except for one entry.  The row of the nonzero entry is 1 if $i=1$ and it is $n_1 +1$ if $i=2$.  The column of the nonzero entry is 1 if $j=1$ and it is $n_1 +1$ if $j=2$.  In this way we get $R$-linear maps $f_{i,j}:E_i S^\rho E_j \to R$ for $i=1,2$ whose images are the ideals $A_{i,j}^\rho$ of $R$ appearing in the proof of Theorem \ref{thm_data}(4).

We note that $\mathcal{A}_{i,j}^\rho=E_i S^\rho E_j$.  The $f_{i,j}$ for $i=1,2$ fit into the following commutative diagram by \cite[Proposition 1.3.8]{bellaiche2009families} (cf. \cite[Corollary 6.5]{berger2020deformations}).
$$
\begin{tikzcd}
\mathcal{A}_{1,2}^\rho \mathcal{A}_{2,1}^\rho \arrow[hook]{r}{} \arrow{d}[swap]{f_{1,2}\otimes f_{2,1}}
               &\mathcal{A}_{1,1}^\rho \arrow{d}{f_{1,1}}\\
A_{1,2}^\rho A_{2,1}^\rho \arrow[hook]{r}{ } &R
\end{tikzcd}
$$

By adaptedness $T(\mathcal{A}_{1,1}^\rho)=\tr \circ P\rho P^{-1}(\mathcal{A}_{1,1}^\rho)$, which equals $f_{1,1}(\mathcal{A}_{1,1}^\rho)$.  Then $f_{1,1}=T$.  With $I=T(\mathcal{A}_{1,2}^\rho \mathcal{A}_{2,1}^\rho)$ from Proposition \ref{prop_image_ideal}, the commutativity of the diagram implies $I=f_{1,1}(\mathcal{A}_{1,2}^\rho\mathcal{A}_{2,1}^\rho) = f_{1,2}\otimes f_{2,1}(\mathcal{A}_{1,2}^\rho \mathcal{A}_{2,1}^\rho) = A_{1,2}^\rho A_{2,1}^\rho$.
\qed
\end{myprop}

\begin{myprop}\label{prop_ideal}
The reducibility ideal $I=A_{2,1}^\rho$.
\proof
Let $\rho'$ be the representation obtained from $\rho$ by factoring through $S^\rho$ and conjugating by some $P \in GL_n(R)$ for adaptedness to $\mathcal{E}^\rho$, the data of idempotents of $S^\rho$ from Theorem \ref{thm_data}(3), using \cite[Proposition 1.3.7]{bellaiche2009families}.  Write the data of idempotents as $e_i$ and $\psi_i$ for $i=1,2$.  For $r\in S^\rho$ write functions $A,B,C,$ and $D$ in
\begin{equation}\label{rho_prime}
\rho'(r) = \begin{bmatrix}
A(r) & B(r) \\
C(r) & D(r)
\end{bmatrix}
\end{equation}
with $A(r)\in M_{n_1}(A^\rho_{1,1})$, $B(r)\in M_{n_1,n_2}(A^\rho_{1,2})$, $C(r)\in M_{n_2,n_1}(A^\rho_{2,1})$, and $D(r)\in M_{n_2}(A^\rho_{2,2})$, where we have applied Theorem \ref{thm_data}(4) to get $A_{i,j}^\rho \subset R$ for $i,j=1,2$.  By adaptedness, for $i=1,2$
$$
\psi_i: e_i S^\rho e_i\overset{\sim}\to M_{n_i}(R)
$$
are lifts of
$$
\tau_i |_{e_i S^\rho e_i}: e_i S^\rho e_i \to M_{n_i}(\F)
$$
and are such that for all $x\in e_i S^\rho e_i$, $T(x)=\tr (\psi_i(x))$.

This implies $\tau_1(r) = A(r) \mod m_R$ and $\tau_2(r) = D(r) \mod m_R$.  Viewing $\rho$ as a representation of $S^\rho$, we have $\rho' \cong \rho$, which implies $\rho' \otimes \F \cong \rho\otimes \F$ for the mod $m_R$ reductions of $\rho'$ and $\rho$.  Therefore, $\rho'\otimes \F$ is a non-split extension of $\tau_2$ by $\tau_1$.  There then exists $r\in R[G]$ such that $\rho' \otimes \F (r) = \begin{bmatrix}
\tau_1(r) & B(r) \\
0 & \tau_2(r)
\end{bmatrix}\mod m_R$ and $B(r)  \neq 0 \mod m_R$.  By Theorem \ref{thm_data}(4),
$$
A_{1,2}^\rho = A_{1,2}^\rho R = A_{1,2}^\rho A_{1,1}^\rho \subset A_{1,1}^\rho=R
$$
and $A_{1,2}^\rho$ is an ideal of $R$.  As $B(r) \in M_{n_1,n_2}(A^\rho_{1,2})$, it must be $A^\rho_{1,2}\not \subset m_R$, which means $A^\rho_{1,2} = R$.  Then the reducibility ideal $I = A^\rho_{1,2}A^\rho_{2,1}=A^\rho_{2,1}$.
\qed
\end{myprop}

\begin{myprop}\label{injective}
There exists an injective $\F$-module homomorphism
$$
\iota : \Hom_R(A^\rho_{2,1},\F) \to \Ext^1_{R[G]/m_R R[G]} (\tau_1, \tau_2) \cong H^1(G, \Hom_\F(\tau_1, \tau_2)).
$$
\proof
We have $I=A_{2,1}^\rho$ by Proposition \ref{prop_ideal}.  By Theorem 1.5.5 of \cite{bellaiche2009families}, putting $J=m_R$, we get the injective map.
\qed
\end{myprop}

Now we establish conditions under which $I$ is a principal ideal.

\begin{myprop}\label{prop_principal}
If $\dim_\F \im \iota \leq 1$ then $I$ is principal.
\proof
Because $\iota$ is injective, $\dim_\F \Hom_R (A^\rho_{2,1},\F)=0$ or $1$.  As $\Hom_R(A^\rho_{2,1},\F ) = \Hom_\F (I/m_R I,\F)$ this space has the same dimension as $I/m_R I$.  By the complete version of Nakayama's lemma, $\dim_\F I/m_R I$ equals the cardinality of a minimal generating set of $I$ as an $R$-module.  If $\dim_\F I/m_R I =0$ then $I=(0)$.  If $\dim_\F I/m_R I=1$ then $I$ has one minimal generator:  $I=(a)$ for a nonzero $a\in R$.  We have shown that $I$ is principal.
\qed
\end{myprop}

\begin{myprop}\label{prop_dvr}
Suppose $I$ is maximal and principal.  Suppose $R$ admits a surjective map to a ring of characteristic zero.  Then $R$ is a discrete valuation ring.
\proof
See also \cite{calegari2006eisenstein} Lemma 3.4.  We know $R/I=\F$ so $I \neq (0)$ as $R$ admits a surjective map to a characteristic zero ring.  Write $I=(a)$ for some $a \in R$.  The generator $a$ is not nilpotent since $R$ is reduced.  We conclude $R$ is a DVR by Chap. 1 Sec. 2, Proposition 2 of \cite{serre1979local}.
\qed
\end{myprop}

\section{Universal deformation ring is a discrete valuation ring}\label{sec:r_dvr}
In this section, we apply our reducibility ideal results to Galois representations.  Let $\F$, $p$, $\Sigma$, $E$, $\mathcal{O}$, and $\varpi$ be as in Section \ref{setup}.  Denote by $\mathcal{C}$ the category of local complete Noetherian $\mathcal{O}$-algebras with residue field $\F$.

Let $\tau_i:G_\Sigma\to GL_{n_i}(\F)$ for $i=1,2$ be absolutely irreducible and non-isomorphic representations.  With $n=n_1 + n_2$ let $\tau:G_\Sigma \to GL_n(\F)$ be a non-semi-simple representation of the form
\begin{equation}\label{eqn_tau}
\tau = \begin{bmatrix}
\tau_1 & *\\
 & \tau_2
\end{bmatrix}.
\end{equation}
Assume that $p\nmid n!$, and that $\tau$ is Fontaine--Laffaille at the primes of $F$ lying over $p$ (see Section \ref{f-l} for the definition; this means $\tau|_{D_v}$ is in the essential image of the Fontaine--Laffaille functor $\mathbf{G}$ for all $v\in \Sigma$ such that $v|p$).  We note that we have used the natural maps $G_{F_v}\to G_F \to G_\Sigma$ to consider $G_\Sigma |_{D_v}$.

Following Mazur an $\mathcal{O}$-deformation of $\tau$ is a pair consisting of $A$ an object of $\mathcal{C}$ and a strict equivalence class of continuous representations $\rho: G_\Sigma \to GL_n(A)$ such that $\tau = \rho \mod m_A$, where $m_A$ is the maximal ideal of $A$.  We denote an $\mathcal{O}$-deformation by a single member of its strict equivalence class, and say it is Fontaine-Laffaille if the representative $\rho$ satisfies $\rho|_{D_v}$ is Fontaine-Laffaille for all $v|p$.

Because $\tau_i$ for $i=1,2$ are absolutely irreducible, they have scalar centralizers.  Consider the deformation problems for $\tau_i$ and $i=1,2$ that for each object $A$ of $\mathcal{C}$ assign the set of strict equivalence classes of representations $\sigma:G_\Sigma \to GL_{n_i}(A)$ that are Fontaine-Laffaille at all $v|p$.  We note that for each $v|p$, the Fontaine--Laffaille deformations satisfy the conditions for a local deformation problem in the sense of Def. 2.2.2 in \cite{clozel2008automorphy} (see also p. 35 in [loc. cit.]).  With scalar centralizers of each, the global deformation problems for each $\tau_i$ are representable by rings in $\mathcal{C}$ following \cite[Proposition 2.2.9]{clozel2008automorphy}.  Denote these by $R_i$ for $i=1,2$.  For places in $\Sigma$ that are not above $p$, we impose no condition.

\begin{myassu}\label{assumptions}
Assume
\begin{enumerate}
\item $\dim_\F H^1_\Sigma(F, \Hom_\F(\tau_2, \tau_1)) =1$, and
\item $R_1=R_2=\mathcal{O}$, which implies that for $i=1,2$ there exists a unique deformation $\tilde \tau_i:G_\Sigma \to GL_{n_i}(\mathcal{O})$ of $\tau_i$ to $\mathcal{O}$.
\end{enumerate}
\end{myassu}

Consider the deformation problem for $\tau$ where to each object $A$ of $\mathcal{C}$ we assign the set of strict equivalence classes of representations $\rho:G_\Sigma \to GL_n(A)$ that are Fontaine-Laffaille.  As $\tau$ is not semi-simple, it has scalar centralizer (cf. \cite[Sec. 6.2]{berger2013deformation}), so there exists a universal deformation ring (by the same argument as for each $\tau_i$), which we denote by $R$, and a universal deformation $\rho^{univ}:G_\Sigma \to GL_n(R)$.  There is a one-to-one correspondence between the set of $\mathcal{O}$-algebra maps $R\to A$ inducing identity on $\F$ and the set of deformations $\rho: G_\Sigma \to GL_n(A)$ of $\tau$.

We write $R^{red}$ for the quotient of $R$ by its nilradical and $\rho^{red}$ for the corresponding universal deformation obtained by composing $\rho^{univ}$ with the natural quotient map $R \twoheadrightarrow R^{red}$.  We view $\rho^{red}$ as a homomorphism of $R$-algebras $\rho^{red}:R^{red}[G_\Sigma]\to M_n(R^{red})$.

Let $T$ denote $\tr\circ \rho^{red}:R^{red}[G_\Sigma]\to R^{red}$ and let $I$ be the ideal of reducibility of $T$.  Let $S$ denote the $R^{red}$-algebra $R^{red}[G_\Sigma]/\ker \rho^{red}$.

By Theorem \ref{thm_data}(4), $S$ has data of idempotents $\mathcal{E} = \left\{ e_i, \psi_i ; i=1,2\right\}$.  The representation $\sigma:S\to M_n(R^{red})$ induced by $\rho^{red}$ can be conjugated by some $P\in GL_n(R^{red})$ to get $P\sigma P^{-1}$ adapted to $\mathcal{E}$.  Denote by $\rho'$ the adapted representation $P\sigma P^{-1}:S\to M_n(R^{red})$.

\begin{myprop}\label{prop_selmer}
There is an injective $\F$-module homomorphism
$$
\iota : \Hom_{R^{red}}(I,\F) \to \Ext^1_{R^{red}[G_\Sigma]/m_{R^{red}} R^{red}[G_\Sigma]} (\tau_1, \tau_2) \cong H^1(G_\Sigma, \Hom_\F(\tau_1, \tau_2)),
$$
and the image of $\iota$ is contained in $H^1_\Sigma (F, \Hom(\tau_1, \tau_2))$.
\proof
The existence and injectivity of $\iota$ follows from Proposition \ref{injective}.  It remains to prove the image lands in the Selmer group.  The GMA $\rho'(S)$ is a standard GMA of type $(n_1,n_2)$ associated to the $R^{red}$-submodules $A^\rho_{i,j}$ of $R^{red}$ by Theorem \ref{thm_data}(4).  Denote the space of $\rho'$ by $V$, which is a free $R^{red}$-module of rank $n$.  The image $\rho'(S)$ consists of matrices that act on $V$.  Define a basis for $V$ by $v_1={}^t[1,0,0,...,0], ..., v_n={}^t[0,...,0,1]$, which are vectors with 1 at one entry and zero elsewhere (if $x$ is a matrix, we denote its transpose by ${}^t x$).  We have $\rho'(s)v_i \in V$ for any $s\in S$, $i=1,...,n$, and $V$ is an $S$-module.  Since $\rho^{red}$ and hence also $\rho'$ is Fontaine--Laffaille, we have that as an $S$-module $V$ is Fontaine--Laffaille.

We use notation $E_i$ as in \cite{bellaiche2009families} Section 1.3.1:  the data $\mathcal{E}$ gives a unique element $E_i \in e_i S e_i$ such that $\psi_i (E_i)$ is the elementary matrix of $M_{n_i}(R^{red})$ with unique nonzero coefficient at row 1 and column 1.  Define $M_i:=SE_i$ for $i=1,2$, which is an $R^{red}$-module and a projective $S$-module.  By $\rho'$ being adapted to $\mathcal{E}$, $\rho'|_{e_1 Se_1 \oplus e_2 S e_2}(E_i)$ equals the image of $\psi_i(E_i)$ under the diagonal map to $M_n(R^{red})$.  For each $s\in S$, $\rho'(s)\rho'(E_i)$ is a matrix of zeroes except one column, corresponding to $E_i$ for $i=1,2$.  By definition of $E_i$, the non-zero column is column 1 (resp. column $n_1+1$) if $i=1$ (resp. $i=2$).  Thus identify rank one matrix $sE_i$ with the column vector, which equals $\rho '(s)v_1$ (resp. $\rho '(s)v_{n_1+1}$) for $i=1$ (resp. $i=2$).  In either case we get $M_i \subset V$.

In particular, $M_1$ is an $S$-submodule of $V$.  The category of Fontaine-Laffaille representations is closed under taking subobjects, quotients, and finite direct sums (\cite[Sec. 2.4.1]{clozel2008automorphy}, \cite[Sec. 1.1.2]{diamond2004tamagawa}).  Because $V$ is Fontaine-Laffaille, we have that $M_1$ is Fontaine-Laffaille.  By Theorem 1.5.6(1) of \cite{bellaiche2009families}, the image of $\iota$ consists precisely of the $S/m_{R^{red}}S$-extensions of $\tau_1$ by $\tau_2$.  By Theorem 1.5.6(2) of \cite{bellaiche2009families}, any $S/m_{R^{red}}S$-extension of $\tau_1$ by $\tau_2$ is a quotient of $M_1/m_{R^{red}} M_1\oplus \tau_2$ by an $S$-submodule.  The quotient $M_1/m_{R^{red}} M_1$ is Fontaine-Laffaille since $M_1$ is Fontaine-Laffaille, and then $M_1/m_{R^{red}} M_1\oplus \tau_2$ is Fontaine-Laffaile as $\tau_2$ is Fontaine-Laffaille.  So we conclude the image of $\iota$ consists of Fontaine-Laffaille extensions.
\qed
\end{myprop}

\begin{mycor}\label{cor_principal}
If $\dim_\F H^1_\Sigma (F, \Hom_\F(\tau_1, \tau_2))\leq 1$ then $I$ is principal.
\proof
Combine Proposition \ref{prop_selmer} with Proposition \ref{prop_principal}.
\qed
\end{mycor}

We will make use of the following result of Berger and Klosin.  They show the $\mathcal{O}$-algebra structure map $\mathcal{O}\to R/I'$ is surjective with Assumption \ref{assumptions} (1) and (2), where $I'$ is the ideal of reducibility of $\tr \rho^{univ}$, and then using a bound on a Selmer group they show $R/I'$ is cyclic.  They then obtain a corresponding statement for $R^{red}/I$.  We take this bound to be sufficiently strong to get maximality of $I$.

Under the assumption that $R_1=R_2=\mathcal{O}$, such as in Assumption \ref{assumptions} (2), we denote by $\tilde{\tau_i}$, $i=1,2$ the unique Fontaine--Laffaille deformation $\tilde{\tau_i}:G_\Sigma \to GL_{n_i}(\mathcal{O})$ with $\tau_i = \tilde{\tau_i} \mod \varpi$.

\begin{myprop}[\cite{berger2013deformation} Proposition 7.14]\label{prop_maximal}
Assume Assumption \ref{assumptions} (1) and (2).  Assume $\# H^1_\Sigma(F,\Hom_\mathcal{O}(\tilde\tau_2, \tilde\tau_1)\otimes E/\mathcal{O}) \leq \# \mathcal{O}/\varpi$.  Then $I$ is maximal.
\end{myprop}

We conclude this section with a theorem identifying conditions when the universal deformation ring is a discrete valuation ring.  We collect the assumptions of this section in the statement.

\begin{mythm}\label{main_thm}
Let $\tau_i:G_\Sigma\to GL_{n_i}(\F)$ for $i=1,2$ be absolutely irreducible and non-isomorphic.  With $n=n_1 + n_2$ let $\tau:G_\Sigma \to GL_n(\F)$ be a non-semi-simple representation of the form
\begin{equation*}
\tau = \begin{bmatrix}
\tau_1 & *\\
 & \tau_2
\end{bmatrix}.
\end{equation*}
Assume that $p \nmid n!$, and that $\tau$ is Fontaine-Laffaille at the primes of $F$ lying over $p$.  Let $(R,\rho^{univ})$ be the universal deformation representing the deformation problem for $\tau$ which assigns the set of strict equivalence classes of Fontaine--Laffaille at $v|p$ representations.  Let $R^{red}$ be the quotient of $R$ by its nilradical and let $\rho^{red}$ be the composition of $\rho^{univ}$ with $R \twoheadrightarrow R^{red}$.  Assume $R_1=R_2=\mathcal{O}$, universal deformation rings classifying deformations of $\tau_i$ that are Fontaine-Laffaille at $v|p$.
Also assume
\begin{enumerate}
\item $\# H^1_\Sigma(F,\Hom_\mathcal{O}(\tilde\tau_2, \tilde\tau_1)\otimes E/\mathcal{O}) \leq \# \mathcal{O}/\varpi$,
\item $\dim_\F H^1_\Sigma(F, \Hom_\F(\tau_1, \tau_2)) \leq 1$, and
\item $R^{red}$ admits a surjective map to a ring of characteristic zero.
\end{enumerate}
Then $R^{red}$ is a discrete valuation ring.
\proof
With $p\nmid n!$, we have that $n!$ is invertible in $\F$, which is required in the definition of pseudocharacter (\cite[Sec. 1.2.1]{bellaiche2009families}).  Let $I$ be the ideal of reducibility of $T=\tr \rho^{red}$.  Using (1) and by Proposition \ref{prop_maximal}, $I$ is maximal in $R^{red}$.  We apply Proposition \ref{prop_selmer} to obtain the injective map $\Hom_{R^{red}}(I,\F) \to H^1_\Sigma(F, \Hom_\F(\tau_1, \tau_2))$, and then use the bound (2) to meet the hypothesis of Corollary \ref{cor_principal}.  This gives us that $I$ is principal.  By Lemma \ref{prop_dvr}, $R^{red}$ is a DVR.
\qed
\end{mythm}

\section{Modularity of 2-dimensional representations of $\Q$ with classical modular forms}\label{sec:2dim}
In this section and the next we show how our results can be used to prove modularity of some Galois representations.  For a congruence subgroup $\Gamma$ of $SL_2(\Z)$ and integer $k\geq 0$, denote the $\C$-space of modular forms of weight $k$ and level $\Gamma$ by $M_k(\Gamma)$.  Denote by $S_k(\Gamma)$ the subspace inside $M_k(\Gamma)$ of cusp forms of weight $k$ and level $\Gamma$.  For each Dirichlet character $\chi$ modulo $N$, we denote by $M_k(N, \chi)$ (resp. $S_k(N,\chi)$) the space of all modular forms (resp. cusp forms) of weight $k$, level $N$, and character $\chi$.

Let $k\geq 4$ be an even integer and let $p>k$ be a prime.  Let
$$
E_k(z) = -\frac{B_k}{2k} + \sum^\infty_{n=1}\sigma_{k-1}(n)q^n, \quad q=e^{2\pi iz},
$$
where $B_k$ is the $k$th Bernoulli number and $\sigma_{k-1}(n)=\sum_{d|n; d>0}d^{k-1}$, be the Eisenstein series of weight $k$ normalized so that its first Fourier coefficient is 1.

Let $\epsilon: G_\Q\to GL_1(\Q_p)$ be the $p$-adic cyclotomic character, and let $\overline \epsilon:G_\Q \to GL_1(\F_p)$ be the mod $p$ cyclotomic character.  Define the Galois representation $\rho_{E_k}:=1\oplus \epsilon ^{k-1}:G_\Q \to GL_2(\Q_p)$.

\begin{mydef}
Let $f=\sum^\infty_{n=0}a_n(f)q^n$ and $g=\sum^\infty_{n=0}a_n(g)q^n$ be two modular forms (each of any level $N$) with Fourier coefficients in the ring of integers $\mathcal{O}_K$ of a number field $K$.  Let $\mathfrak{p}\subset \mathcal{O}_K$ be a prime ideal.  We say $f$ and $g$ are \textit{congruent modulo $\mathfrak{p}$} and write $f \equiv g \mod \mathfrak{p}$ if $a_n(f)-a_n(g) \equiv 0 \mod \mathfrak{p}$ for all $n>0$.
\end{mydef}

Let $\Sigma$ be a finite set of rational primes of $\Q$ containing $p$.  Let $G_\Sigma$, $\F$, $E$, $\mathcal{O}$, and $\varpi$ be as in Section \ref{setup}. For an eigenform $g \in S_k(N, \psi)$, denote by $\rho_g$ the irreducible Galois representation attached to $g$.  Denote the $G_\Sigma$-modules $\Hom_\mathcal{O}(\epsilon^{k-1}, 1)$, $\Hom_\mathcal{\F}(\overline{\epsilon}^{k-1}, 1)$, $\Hom_\mathcal{O}(1,\epsilon^{k-1})$, and $\Hom_\mathcal{\F}(1,\overline{\epsilon}^{k-1})$ by $\epsilon^{1-k}$, $\overline\epsilon^{1-k}$, $\epsilon^{k-1}$, and $\overline\epsilon^{k-1}$, respectively.  

\subsection{Residual representations}
Suppose $G$ is $G_\Q$ or $G_\Sigma$ (or the analogous quotients of $G_K$ in Section \ref{sec:4dim}), and suppose $\rho: G \to GL_n(E)$ is a continuous representation.  We denote by $\overline \rho$ the representation into $GL_n(\F)$ obtained by reducing $\rho$ modulo $\varpi$ with respect to some lattice in $E^n$.  We note that the isomorphism class of $\overline \rho$ depends on the choice of lattice, but the semi-simplification denoted $\overline \rho ^{ss}$ does not.

\subsection{Modularity theorem}
\begin{myprop}
Assume each $\ell \in \Sigma - \left\{ p\right\}$ satisfies $\ell \not \equiv \textnormal{1 (mod $p$)}$.  Then the Fontaine--Laffaille universal deformation rings classifying deformations of $1$, $\overline{\epsilon}^{k-1}$ equal $\mathcal{O}$.
\proof
We follow the argument in the proof of \cite[Lemma 7.2]{berger2020deformations}.  It suffices to show that the only infinitesimal deformations satisfying the Fontaine--Laffaille property of $1$ and $\overline{\epsilon}^{k-1}$ to $\F[X]/X^2$ are the trivial ones.  Suppose $\overline{\epsilon}^{k-1} + \alpha X$ is a deformation of $\overline{\epsilon}^{k-1}$ to $\F[X]/X^2$.  Since $\overline{\epsilon}^{1-k}$ is Fontaine--Laffaille, we have that $1+\overline{\epsilon}^{1-k}\alpha X$ is Fontaine--Laffaille, where $\overline{\epsilon}^{1-k}\alpha$ is a homomorphism from $G_\Sigma$ to the additive group $\F$.  Let $\ell\in \Sigma - \left\{ p \right\}$.  Since $\ell \not \equiv \textnormal{1 (mod $p$)}$, by local class field theory $\overline{\epsilon}^{1-k}\alpha:G_\Sigma \to \F^+$ is unramified at $\ell$.  This means $1+\overline{\epsilon}^{1-k}\alpha X : G_\Sigma \to GL_1(\F[X]/X^2)$ can only be ramified at $p$.  We apply Lemma 9.6 in \cite{berger2013deformation} to $1+\overline{\epsilon}^{1-k}\alpha X$, replacing the imaginary quadratic field there with $\Q$ for us (as done in Lemma 7.2 in \cite{berger2020deformations}) and crystallinity there with Fontaine--Laffaille for us (cf. end of Section 5 in \cite{berger2013deformation}).  Then we have that $1+\overline{\epsilon}^{1-k}\alpha X$ is unramified at $p$, and thus unramified everywhere.  With $p\nmid \# Cl_\Q$, it follows that $\alpha=0$, so $\overline{\epsilon}^{k-1}$ only has trivial infinitesimal Fontaine--Laffaille deformations.  A similar argument shows the only deformation of 1 to $\F[X]/X^2$ is the trivial one.
\qed
\end{myprop}

\begin{mythm}\label{thm_2dim}
Let $k\geq 4$ be an even integer and let $p>k$ be a prime.  Let $\Sigma$, $E$, $\mathcal{O}$, $\varpi$, and $\F$ be as above.  Assume that
\begin{enumerate}[(i)]
\item if $\ell \in \Sigma - \left\{ p\right\}$ then $\ell \not \equiv \textnormal{1 (mod $p$)}$,
\item $\# H^1_\Sigma (\Q, \epsilon^{1-k}\otimes E/\mathcal{O}) \leq \# \F$, and
\item $\dim_\F H^1_\Sigma (\Q, \overline\epsilon^{k-1}) \leq 1$.
\end{enumerate}
Suppose $\rho: G_\Sigma \to GL_2(E)$ is an irreducible representation that is Fontaine--Laffaille at $p$, such that $\overline\rho^{ss}\cong 1\oplus \overline{\epsilon}^{k-1}$.  Then one has the following.
\begin{enumerate}
\item If $p|B_k$, then $\rho\cong \rho_{g}$ where $g\in S_k(SL_2(\Z))$ (i.e. $\rho$ is modular).
\item If $p>k+1$ and if there is $\ell \in \Sigma - \left\{ p\right\}$ such that $\ell^k \equiv \textnormal{1 (mod $p$)}$, then $\rho\cong \rho_{g}$ where $g \in S_k(\Gamma_0(\ell))$ is a newform (i.e. $\rho$ is modular).
\end{enumerate}

\proof
In case (1), the divisibility of $B_k$ by $p$ implies there exists a normalized eigenform $g_1 \in S_k(SL_2(\Z))$ that is a newform such that $g_1\equiv E_k \mod \mathfrak{p_1}$ where $\mathfrak{p_1}$ is a prime of a number field with $\mathfrak{p_1}|p$ (see for example \cite[Proposition 1]{ghate2002introduction}).  In case (2), the divisibility $p\mathrel | \ell^k -1$ implies there exists a newform $g_2 \in S_k(\Gamma_0(\ell))$ such that $g_2\equiv E_k \mod \mathfrak{p_2}$ where $\mathfrak{p_2}$ is a prime of a number field with $\mathfrak{p_2}|p$ (by \cite[Theorem 1]{billerey2016modularity}).  The remainder of the proof is the same for $g_i$, $i=1,2$, so we denote either one by $g$.  By a result of Ribet \cite[Proposition 2.1]{ribet1976modular}, there exists a $G_\Sigma$-stable $\mathcal{O}$-lattice in the space of $\rho_g$ such that
\begin{equation}\label{eqn_g}
\overline{\rho_g} = \begin{bmatrix}
1 & *\\
 & \overline{\epsilon}^{k-1}
\end{bmatrix}
\end{equation}
and this extension is not split.  The fact that the level of $g$ is prime to $p$, together with $2\leq k < p$, implies $\rho_{g}$ is Fontaine-Laffaille at $p$ (see e.g. \cite[Sec. 1]{dousmanis2010}).  Because $\rho_g$ with a lattice as in equation (\ref{eqn_g}) is Fontaine--Laffaille, and the essential image of the Fontaine-Laffaille functor is closed under quotients (\cite[Sec. 2.4.1]{clozel2008automorphy}), $\overline {\rho_{g}}$ is Fontaine-Laffaille at $p$.  The representation $\overline {\rho_{g}}$ has scalar centralizer (cf. \cite[Sec. 6.2]{berger2013deformation}).

Let $\mathcal{C}$ denote the category of local complete Noetherian $\mathcal{O}$-algebras with residue field $\F$.  The deformation problem for $\overline {\rho_{g}}$ where for each object $A$ of $\mathcal{C}$ we assign the set of strict equivalence classes of representations $\rho:G_\Sigma \to GL_2(A)$ that are Fontaine-Laffaille at $p$ is representable by a ring $R$ and there is a universal deformation $\rho^{univ}:G_\Sigma \to GL_2(R)$.  Let $R^{red}$ be the quotient of $R$ by its nilradical and let $\rho^{red}$ be the composition of $\rho^{univ}$ with $R \twoheadrightarrow R^{red}$.

Because $\rho_{g}$ with a lattice as in equation (\ref{eqn_g}) is a deformation of $\overline{ \rho_{g}}$, universality implies there is an $\mathcal{O}$-algebra map $R^{red} \to \mathcal{O}$, which must be surjective (and $\mathcal{O}$ has characteristic zero).  With the conditions on the Selmer groups in the statement we apply Theorem \ref{main_thm} to get $R^{red}$ is a DVR, which implies $R^{red}\cong \mathcal{O}$.  So $\rho^{red}\cong\rho_{g}$ and $\rho^{red}$ is modular.  Again by Ribet's lemma \cite[Proposition 2.1]{ribet1976modular} there is a $G_\Sigma$-stable lattice in which $\rho$ is valued and $\overline{\rho}$ is not semi-simple.  The class of the extension of $\overline{\epsilon}^{k-1}$ by 1 determining $\overline{\rho}$ must be the same as $\overline{\rho_g}$, because there is only one by (ii). So $\rho$ is also a deformation of $\overline{\rho_g}$, and the conclusion is that $\rho\cong \rho_{g}$.
\qed
\end{mythm}

\begin{myrem}\label{rem:2dim}
The modularity of $\rho$ in Theorem \ref{thm_2dim} is already known.  Assume the hypotheses of Theorem \ref{thm_2dim} and let $\rho' : G_\Q \to GL_2(E)$ be irreducible, Fontaine--Laffaille at $p$, and $\overline{\rho'}^{ss}\cong 1\oplus \overline{\epsilon}^{k-1}$.  Also assume $\rho'$ is unramified outside $\Sigma$.  It is known that $\rho ' |_{D_p}$ has distinct Hodge-Tate weights: $\rho ' |_{D_p}$ is, upto a twist by a unique power of $\epsilon$, crystalline with Hodge-Tate weights $\left\{ 0,r-1 \right\}$ for some integer $r\geq 2$ (cf. Sec. 3.1 of \cite{breuil2003quelques}); and that $\rho'$ is odd (cf. \cite[Proposition 1.2]{berger2018oddness}).  Then the Fontaine-Mazur conjecture in the residually reducible case (\cite[Theorem 1.0.2]{pan2022fontaine} and \cite{skinner1999residually}) implies $\rho'$ arises from a cuspidal eigenform.  As $\rho'$ is unramified outside $\Sigma$, it factors through $\rho':G_\Sigma \to GL_2(E)$.  One may view Theorem \ref{thm_2dim} as an alternative method for a modularity result in a specific case.

\end{myrem}

\begin{myrem}
If we consider $\overline\rho = \left[\begin{smallmatrix}1 & *\\  & \overline{\epsilon}^{k-1}\end{smallmatrix}\right]$ not split as in the proof of Theorem \ref{thm_2dim}, the analogue of Serre's modularity conjecture for reducible, odd residual representations proved in many cases \cite{hamblen2008deformations, fakhruddin2020lifting} implies $\overline \rho$ arises from a modular form of weight $k$ and level prime to $p$.  The proof of this by Fakhruddin, Khare, and Patrikis \cite[Theorem 7.4]{fakhruddin2020lifting} first produces a lift to characteristic zero with sufficient conditions to apply modularity results of \cite{skinner1999residually} and \cite{pan2022fontaine}.
\end{myrem}

\subsection{Examples in prime level}
Let $k\geq 4$ be an even integer and let $p>k+1$ be a prime.  Suppose $p$ is a regular prime and $\ell$ is a prime such that $p\mathrel\Vert \ell^k -1$.  Put $\Sigma=\left\{ p, \ell\right\}$.

\begin{myprop}
The set $\Sigma$ satisfies condition (ii) of Theorem \ref{thm_2dim}.
\proof
Put $\eta_k:= B_k(1-\ell^k)$.  Since $p$ is a regular prime with $p\mathrel\Vert \ell^k -1$, it follows that $\ord_\varpi (\eta_k) = \ord_\varpi ((1-\ell^k))=1$ and $\# \mathcal{O}/\eta_k=\# \F$.  By \cite[Proposition 5.7]{berger2019modularity}, one has $\# H^1_{\Sigma} (\Q, \epsilon^{1-k}\otimes E/\mathcal{O}) \leq \# \mathcal{O}/\eta_k$.
\qed
\end{myprop}

\begin{myprop}
If $p\nmid (\ell +1)$ and $p\nmid (\ell -1)$, then the set $\Sigma$ satisfies conditions (i) and (iii) of Theorem \ref{thm_2dim}.
\proof
Condition (i) is immediate and $\ell^2 \not \equiv 1 \mod p$.  We now consider condition (iii).  Since $p \mathrel| (\ell^{k-2}-1)$ would imply $\ell^k = \ell^2 \ell^{k-2}\not \equiv 1 \mod p$, contrary to $p\mathrel |(\ell^k -1)$, we get $p \mathrel \nmid (\ell^{k-2}-1)$.  Then $H^1_\Sigma(\Q, \overline{\epsilon}^{k-1})= H^1_{\left\{ p \right\}}(\Q, \overline{\epsilon}^{k-1})$ by \cite[Lemma 6.3]{berger2019modularity} (where we put $n=1-k$ and note that $\ord_p (\ell^{k-2}-1)=\ord_p(\ell^{2-k})$).  We apply \cite[Proposition 6.5]{berger2019modularity}, using $p$ is regular and $k-1 \not \equiv 1 \mod (p-1)$, to get $\dim_{\F} H^1 (G_{\left\{ p \right\} },\overline{\epsilon}^{k-1}) \leq 1$, and note that the Selmer group is a subspace.
\qed
\end{myprop}

With $p,k, \ell,$ and $\Sigma$ as at the start of this subsection with $p\nmid (\ell +1)$ and $p\nmid (\ell -1)$, we apply Theorem \ref{thm_2dim}.  The conclusion is that if $\rho:G_{\Sigma} \to GL_2(E)$ is irreducible, Fontaine--Laffaille at $p$, and $\overline\rho^{ss}\cong 1\oplus \overline{\epsilon}^{k-1}$, then $\rho\cong \rho_{f}$ where $f \in S_{k}(\ell)$ (i.e. $\rho$ is modular).  We note $\ell <p$ implies $p\nmid (\ell +1)$ and $p\nmid (\ell -1)$.  For $p<30$ and $\ell <p$, we provide the following explicit examples of $\ell$ and $k$ when these conditions hold.  When $p=13$, there is $S_6(3)$, $S_4(5)$, and $S_8(5)$.  When $p=17$, there is $S_8(2)$, $S_4(13)$, and $S_8(13)$.  When $p=19$, there is $S_6(7)$ and $S_6(11)$.  When $p=29$, there is $S_4(17)$.

\subsection{Example of $p|B_k$ and level 1}
Let $p=691$, $k=12$, and $\Sigma = \left\{ 691 \right\}$.  Note $p=691\nmid B_{680}$ since 691 is prime with irregular index 2, $691 | B_{12}$, and $691 | B_{200}$ (see for example $A073277$ of \cite{sloane2003line} and \cite[p. 32]{lehmer1954application}).

It suffices that the coefficient field $E$ be $\Q_p$ since the unique normalized newform $\Delta \in S_{12}(SL_2(\Z))$ has integer Fourier coefficients and $\rho_\Delta: G_\Q \to GL_2(\Q_{691})$.  Let $\F_{691}:=\Z_{691}/691$.  Now $\epsilon$ is the 691-adic cyclotomic character and $\overline{\epsilon}$ is the mod $691$ cyclotomic character.  We consider the $G_{\left\{ 691\right\}}$-modules ${\epsilon}^{-11}$, $\overline{\epsilon}^{-11}$, ${\epsilon}^{11}$ and $\overline{\epsilon}^{11}$.

\begin{myprop} One has $\dim_{\F_{691}} H^1_{\left\{691 \right\}} (\Q, \overline{\epsilon}^{11}) \leq 1.$
\proof
As $691 \nmid B_{680}$ it follows that the $\overline{\epsilon}^{1-680}$-eigenspace of the 691-part of $Cl_{\Q(\mu_{691})}$ is trivial by the Herbrand--Ribet theorem \cite{ribet1976modular}.  Since $1-680 \equiv 11 \mod 690$, the $\overline{\epsilon}^{11}$-eigenspace of the $691$-part of $Cl_{\Q(\mu_{691})}$ is trivial.  By \cite[Proposition 6.5]{berger2019modularity}, with $k-1 = 11 \not \equiv 1 \mod (p-1)$, we know the second inequality in
$$
\dim_{\F_{691}} H^1_{\left\{691 \right\}} (\Q, \overline{\epsilon}^{11}) \leq \dim_{\F_{691}} H^1 (G_{\left\{ 691 \right\} },\overline{\epsilon}^{11}) \leq 1
$$
holds and the first inequality follows from the Selmer group being a subspace.
\qed
\end{myprop}

\begin{myprop} One has $\# H^1_{\left\{ 691 \right\}} (\Q, \epsilon^{-11}\otimes \Q_{691}/\Z_{691}) \leq \# \F_{691}.$
\proof
As in the proof of Proposition 5.7 in \cite{berger2019modularity} and using the fact that $691 \mathrel\Vert B_{12}$, the inequality follows from the main conjecture of Iwasawa theory proved by Mazur and Wiles.
\qed
\end{myprop}

Applying Theorem \ref{thm_2dim} the conclusion is that if $\rho:G_{\left\{ 691\right\}} \to GL_2(\Q_{691})$ is irreducible, Fontaine--Laffaille at $691$, and $\overline\rho^{ss}\cong 1\oplus \overline{\epsilon}^{11}$, then $\rho\cong \rho_{\Delta}$ where $\Delta \in S_{12}(SL_2(\Z))$ (i.e. $\rho$ is modular).

\section{Modularity of 4-dimensional representations of $\Q(i)$ with Hermitian modular forms}\label{sec:4dim}
Fix $K=\Q(i)$.
\begin{myrem}
We restrict the exposition to the special case of $K=\Q(i)$ for the sake of simplicity of presentation, but Theorem \ref{thm_4dim} also applies to imaginary quadratic fields with odd discriminant.  Congruences between Hermitian modular forms associated to unitary groups in this general case are considered in \cite[Theorem 8.10]{klosin2015maass}.
\end{myrem}

\subsection{Hermitian modular forms}
Denote the Hermitian upper-half plane by
$$
\mathcal{H}_2 = \left\{ Z\in M_2(\C): -I (Z-Z^*)>0 \right\},
$$
where $I = \left[ \begin{smallmatrix} i & 0 \\ 0 & i \end{smallmatrix} \right]$ and $Z^*$ is the conjugate transpose of $Z$.  Let
$$
G_\mu:=GU(2,2) = \left\{ A \in Res_{K/\Q} {GL_4}_{/K} : AJA^* = \mu(A)J\right\}
$$
be the unitary group associated to $K$ over $\Q$ where $J = \left[ \begin{smallmatrix}  & -I_2 \\ I_2 &  \end{smallmatrix} \right]$ with $I_2$ the $2\times 2$ identity matrix, $Res_{K/\Q}$ the Weil restriction, and $\mu(A) \in \mathbf{G}_m$.  Let
$$
G_\mu^+(\R) := \left\{ g \in G_\mu(\R) : \mu(g)>0\right\}
$$
and
$$
G := U(2,2)=\left\{ g \in G_\mu(\R) : \mu(g)=1\right\}.
$$
There is an action for $\gamma \in G_\mu^+(\R)$ on $\mathcal{H}_2$ defined by
$$
\gamma Z = (a_\gamma Z + b_\gamma) (c_\gamma Z + d_\gamma)^{-1}
$$
where $\gamma = \left[ \begin{smallmatrix} a_\gamma & b_\gamma \\ c_\gamma & d_\gamma \end{smallmatrix}\right] \in G_\mu^+(\R)$.  For a holomorphic function $F$ on $\mathcal{H}_2$, an integer $m$ and $\gamma \in G_\mu^+ (\R)$ put
$$
F|_m \gamma (Z)= \mu(\gamma)^{2m-4}j(\gamma, Z)^{-m}F(\gamma Z),
$$
with $j(\gamma,Z)=\det(c_\gamma Z + d_\gamma)$.  Define the \textit{Hermitian modular group} associated to $K$ as
$$
\Gamma_\Z :=\left\{ A\in GL_4(\mathcal{O}_K) : AJA^* =J\right\}.
$$
There are congruence subgroups of $G(\Q)$ analogous to $\Gamma_0(N), \Gamma_1(N)$ and $\Gamma(N)$ of $SL_2(\Z)$ \cite[Sec. 2.2]{klosin2009congruences}.  For a congruence subgroup $\Gamma^h$ of $\Gamma_\Z$, a holomorphic function $F$ on $\mathcal{H}_2$ is a \textit{Hermitian modular form} of weight $m$ and level $\Gamma^h$ if
$$
F|_m \gamma = F \quad \text{ for all } \gamma \in \Gamma^h.
$$
Denote the $\C$-vector space of Hermitian modular forms of weight $m$ and level $\Gamma_\Z$ by $\mathcal{M}_m(\Gamma_\Z)$.  A Hermitian modular form of level $\Gamma_\Z$ has a Fourier expansion
$$
F(Z)=\sum_{\tau \in \mathcal{S}}c(\tau)exp(2\pi i \; {\tr \tau Z}) \quad \text{ with $c(\tau) \in \C$}
$$
where $\mathcal{S}=\left\{ \left[\begin{smallmatrix} s_1 & s_2 \\ \overline{s_2} & s_3\end{smallmatrix}\right] \in M_2(K): s_1, s_3 \in \Z, s_2 \in \frac{1}{2}\mathcal{O}_K\right\}$.  Denote by $\mathcal{S}_m(\Gamma_\Z)$ the subspace inside $\mathcal{M}_m(\Gamma_\Z)$ of \textit{cusp forms} of weight $m$ and level $\Gamma_\Z$ (see \cite[Sec. 2.3]{klosin2009congruences}).

\subsection{Notation}
Here we introduce some notation used by \cite{klosin2009congruences}.  Let $k$ be a positive integer divisible by 4 and $p>k$ a prime.  Let $\Sigma$ be a finite set of finite primes of $K$ containing $\Sigma_p$, the set of primes of $K$ over $p$.  Let $G_\Sigma$, $\F$, $E$, $\mathcal{O}$, and $\varpi$ be as in Section \ref{setup}.  Denote by $\ord_\varpi$ the $\varpi$-adic valuation.  For $z \in \C$, denote by $\overline z$ the complex conjugate of $z$.

\begin{enumerate}
\item Let $\mathcal{N}$ denote a basis of $S_{k-1}(4,(\frac{-4}{.}))$ consisting of primitive (newforms) normalized eigenforms.  For $f=\sum^\infty_{n=1}a_n q^n \in \mathcal{N}$, set $f^\rho := \sum^\infty_{n=1}\overline{a_n}q^n$.
\item For a Hecke character $\chi$ of $K$, we denote by ${\cond \chi}$ the conductor of $\chi$.
\item Let $f\in\mathcal{N}$ and $\psi$ a Hecke character of $K$.  We will use the $L$-function $L(BC(f), s, \psi)$ (see Def. 6.1 in [loc. cit.]) and the symmetric square $L$-function $L({\Symm^2 f},s)$ (see (4.9) in [loc. cit.]).  Then assuming $f$ is ordinary at $p$, $L^{int}(BC(f), s, \psi)$ (resp. $L^{int}({\Symm^2 f},s)$) is defined by dividing $L(BC(f), s, \psi)$ (resp. $L({\Symm^2 f},s)$) by a suitable period (see (7.20) in [loc. cit.]).
\item The Maass space is the $\C$-linear subspace of $\mathcal{S}_k(\Gamma_\Z)$ consisting of $F\in \mathcal{S}_k(\Gamma_\Z)$ whose Fourier expansion satisfies suitable conditions (see Sec. 4.1 in [loc. cit.]).  For any normalized Hecke eigenform $f\in S_{k-1}(4, (\frac{-4}{.}))$ there is a lifting $F_f$ in the Maass space satisfying $F_f \neq 0$ only if $f\neq f^\rho$.  When $f\neq f^\rho$, $F_f$ is called the \textit{Maass lift} of $f$ or \textit{CAP lift} of $f$ (Def. 4.4 in [loc. cit.]).  The orthogonal complement of the Maass space with respect to the Petersson inner product consists of non-CAP forms.
\item The Hermitian Hecke algebra $\T^h_\C$ is defined as a certain subalgebra of $End_\C(\mathcal{S}_k(\Gamma_\Z))$, and $F \in \mathcal{S}_k(\Gamma_\Z)$ is an eigenform if it is an eigenfunction for all $T\in\T^h_\C$.  One defines $\Z$-subalgebras $\T^h_\Z$ and $\T^h_\mathcal{O}$ of $\T^h_\C$ in a natural way.  For $F\in \mathcal{S}_k(\Gamma_\Z)$ an eigenform denote by $\lambda_{F,\C}(T)$ the eigenvalue of $T \in \T^h_\C$ corresponding to $F$.  There exists a number field $L_F$ with ring of integers $\mathcal{O}_{L_F}$ such that $\lambda_{F,\C}(T) \in \mathcal{O}_{L_F}$ for all $T \in \T^h_{\mathcal{O}_{L_F}}$ (Theorem 5.9 in [loc. cit.]).  We choose $E$ to contain $L_F$ for all $F$ in a basis of eigenforms.  To each eigenform $F$ there is an $\mathcal{O}$-algebra morphism $\lambda_F : \T^h_\mathcal{O}\to \mathcal{O}$ assigning to $T$ the eigenvalue of $T$ corresponding to $F$.
\end{enumerate}

\subsection{Congruence result of Klosin}\label{sec_congruence}
We now describe Theorem 7.12 and Corollary 7.17 in \cite[Section 7.3]{klosin2009congruences}.  Let $f\in \mathcal{N}$ be ordinary at $p$ and such that $\overline{\rho_f}|_{G_K}$ is absolutely irreducible.  We fix a Hecke character $\chi$ of $K$ under certain conditions.\footnote{Fix a positive integer $m\in S_{f,p}$ (see Definition 7.11 in \cite{klosin2009congruences}).  Fix $\chi$ such that $\ord_p ({\cond \chi})=0$ and $\chi_\infty (z) = \left(\frac{z}{|z|}\right)^{-t}$, with $-k\leq t < -6$ and $p\nmid \# X_{-t-2,4mN_{K/\Q}({\cond \chi})}$ (see Sec. 7.1 in [loc. cit.]).}

If
\begin{align*}
-n:= {\ord_\varpi} &\left(\prod^2_{j=1}L^{int}(BC(f), j+(t+k)/2,\overline{\chi \omega}) \right) \\
&- \ord_\varpi (L^{int}({\Symm^2 f},k)) <0
\end{align*}

where $\omega$ is the unique Hecke character of $K$ which is unramified at all finite places and such that
$$
\omega_\infty (z) = \left(\frac{z}{|z|}\right)^{-k}
$$
then there exists a non-CAP cuspidal Hecke eigenform $F$ such that
\begin{equation}\label{eqn_congruence}
\ord_\varpi (\lambda_{F_f}(T^h) - \lambda_{F}(T^h)) > 0
\end{equation}
for all Hecke operators $T^h \in \T^h_\mathcal{O}$.

\subsection{Galois representations attached to Hermitian modular forms}

\begin{mythm}\label{thm_galois_hermitian}
Let $F \in \mathcal{S}_k (\Gamma_\Z)$ be an eigenform.  There exists a finite extension $E_F$ of $\Q_p$ and a 4-dimensional semi-simple Galois representation $\rho_F : G_K \to GL_{E_F}(V)$ unramified away from the primes of $K$ dividing $2p$ and such that:
\begin{enumerate}
\item For any prime $\mathfrak{q}$ of $K$ such that $\mathfrak{q} \nmid 2p$, the set of eigenvalues of $\rho_F(\Frob_\mathfrak{q})$ coincides with the set of Satake parameters of $F$ at $\mathfrak{q}$.
\item If $\mathfrak{q}$ is a place of $K$ over $p$, the representation $\rho_F|_{D_\mathfrak{q}}$ is crystalline.
\item For $p>k$, and $\mathfrak{q}$ a place of $K$ over $p$, the representation $\rho_F|_{D_\mathfrak{q}}$ is Fontaine--Laffaille.
\end{enumerate}
\proof
This theorem appears as \cite[Theorem 9.2]{klosin2009congruences} (see also \cite[Theorem 6.7]{klosin2015maass}).  It is proved in \cite[Theorem B]{skinner2012galois}.
\qed
\end{mythm}

\subsection{Modularity theorem}
We assume the hypotheses of the main congruence result of Section \ref{sec_congruence}.  In particular $k$ and $p$ are as above.  Under these conditions, the eigenform $f\in \mathcal{N}$ satisfies $f\neq f^\rho$ (by \cite[Proposition 8.13]{klosin2009congruences} with $\overline{\rho_f}|_{G_K}$ absolutely irreducible) and $F_f \in \mathcal{S}_k(\Gamma_\Z)$ is also an eigenform to which we attach the Galois representation $\rho_{F_f}$ over $E$.  The standard $L$-function of $F_f$ factors \cite[Proposition 6.4]{klosin2009congruences}, which implies
$$
\rho_{F_f} =
\begin{bmatrix}
\rho_f|_{G_K}  &  \\
 & (\rho_f \otimes \epsilon)|_{G_K}
\end{bmatrix}
$$
because part (1) of Theorem \ref{thm_galois_hermitian} implies the factors at $\mathfrak{q}$ coincide:  $L(\rho_{F_f},s)_\mathfrak{q} = L_{st}(F_f,s)_\mathfrak{q}$, which relates Hecke eigenvalues to the characteristic polynomial and eigenvalues of $\rho_{F_f}(\Frob_\mathfrak{q})$ for all $\mathfrak{q}\nmid 2p$.  Let $F$ be a non-CAP cuspidal Hecke eigenform with eigenvalues congruent to those of $F_f$ as in (\ref{eqn_congruence}) of Section \ref{sec_congruence}.  The congruence implies
$$
\tr \rho_F (\Frob_\mathfrak{q}) \equiv \tr\rho_{F_f}(\Frob_\mathfrak{q}) \mod \varpi \quad \quad \text{for all $\mathfrak{q}\nmid 2p$}
$$
by part (1) of Theorem \ref{thm_galois_hermitian}.  Since $\left\{ \Frob_\mathfrak{q} : \mathfrak{q} \nmid 2p \right\}$ is dense in $G_K$ by the Chebotarev Density Theorem, we obtain $\tr \overline{\rho_F}  = \tr \overline{\rho_{F_f}}$, and then $\overline{\rho_F}^{ss} \cong \overline{ \rho_{F_f}}^{ss}$ by the Brauer--Nesbitt Theorem.  Since $\rho_{F_f}$ is split, $\overline{\rho_{F_f}}$ is split and then $\overline{\rho_{F_f}} \cong \overline{\rho_{F_f}}^{ss}$.  Hence $\overline{\rho_F}^{ss} \cong \overline{\rho_f}|_{G_K} \oplus (\overline{(\rho_f \otimes \epsilon)}|_{G_K})$.

We view $\rho_f$, $\rho_F$ and $\rho_{F_f}$ as representations of $G_\Sigma$.  Set $\tau_1 := \overline{\rho_f}|_{G_\Sigma}$, and $\tau_2 := \overline{(\rho_f\otimes \epsilon)}|_{G_\Sigma}$.  Set $\tilde \tau_1 := \rho_f|_{G_\Sigma}$, and $\tilde\tau_2 := (\rho_f\otimes \epsilon)|_{G_\Sigma}$.

\begin{mythm}\label{thm_4dim}
Let $K=\Q(i)$.  Let $k$ be a positive integer divisible by 4 and $p>k$ a prime.  Let $f\in \mathcal{N}$ be ordinary at $p$ and such that $\overline{\rho_f}|_{G_K}$ is absolutely irreducible.  Fix $\chi$ and $\omega$ Hecke characters of $K$ as in Section \ref{sec_congruence}.  Assume that for each non-CAP cuspidal Hecke eigenform $F$ congruent to $F_f$ as in (\ref{eqn_congruence}) of Section \ref{sec_congruence}, the representation $\rho_F$ is irreducible.  Also assume
\begin{enumerate}
\item the Fontaine--Laffaille universal deformation rings for $\tau_1,\tau_2$ equal $\mathcal{O}$,
\item $\# H^1_\Sigma(K,\Hom_\mathcal{O}(\tilde\tau_2, \tilde\tau_1)\otimes E/\mathcal{O}) \leq \# \mathcal{O}/\varpi$, and
\item $\dim_\F H^1_\Sigma(K, \Hom_\F(\tau_1, \tau_2))\leq1$.
\end{enumerate}
Then if
\begin{align*}
-n:= \ord_\varpi & \left(\prod^2_{j=1}L^{int}(BC(f), j+(t+k)/2,\overline{\chi \omega}) \right) \\
&- \ord_\varpi (L^{int}({\Symm^2 f},k)) <0,
\end{align*}
$\rho: G_\Sigma \to GL_4(E)$ is irreducible, Fontaine-Laffaille at all $\mathfrak{q}|p$, and $\overline{\rho}^{ss}\cong \tau_1 \oplus \tau_2$, then $\rho \cong \rho_F$ where $F\in \mathcal{S}_k(\Gamma_\Z)$ (i.e. $\rho$ is modular).
\proof
As above $\overline{\rho_F}^{ss} \cong \tau_1 \oplus \tau_2$.  We use a higher-dimensional analogue of Ribet's lemma as in \cite[Proposition 8.1]{berger2013deformation}, which is a special case of a theorem of Urban, to get a lattice inside the space of $\rho_F$ such that
\begin{equation}\label{eqn_rhof}
\overline {\rho_F} = \begin{bmatrix}
\tau_1  &  *\\
 & \tau_2
\end{bmatrix}
\end{equation}
and is not semi-simple.

Consider the deformation problem that for each object $A$ of $\mathcal{C}$, the category of local complete Noetherian $\mathcal{O}$-algebras with residue field $\F$, assigns the set of strict equivalence classes of Fontaine-Laffaille at $v \in \Sigma_p$ deformations of $\overline{\rho_F}$ to $A$, which results in a pair $(\rho^{univ},R)$ representing the deformation problem.  Let $R^{red}$ be the quotient of $R$ by its nilradical and $\rho^{red}:G_\Sigma \to GL_4(R^{red})$ be the composition of $\rho^{univ}$ with $R\twoheadrightarrow R^{red}$.  By Theorem \ref{thm_galois_hermitian} (2) and (3), $\rho_F$ is Fontaine-Laffaille at all $\mathfrak{q}|p$.  Because $\rho_F$ with a lattice as in equation (\ref{eqn_rhof}) is a deformation of $\overline{\rho_F}$ to $\mathcal{O}$, there is a surjective $\mathcal{O}$-algebra map from $R\to \mathcal{O}$ factoring through $R^{red}$.  Theorem \ref{main_thm} implies $R^{red}$ is a DVR, which gives $R^{red}\cong \mathcal{O}$.  Then $\rho^{red} \cong \rho _F$ and $\rho^{red}$ is modular.  Applying \cite[Proposition 8.1]{berger2013deformation}, which requires that $\rho$ be irreducible, there is a $G_{\Sigma}$-stable lattice in which $\rho$ is valued and $\overline{\rho}$ is not semi-simple.  The extension class of $\overline{\rho}$ must be that of $\overline{\rho_F}$ because there is only one by (2), so $\rho$ is a deformation of $\overline{\rho_{F}}$ to $\mathcal{O}$.  We have $\rho \cong \rho^{red}\cong \rho_{F}$ and $\rho$ is modular.
\qed
\end{mythm}

\begin{myrem}
The hypotheses that $f \neq f^\rho$ and $\overline{\rho_f}|_{G_K}$ is absolutely irreducible are used in the method for constructing the Maass lift and congruence.  The ordinarity assumption on the eigenform $f\in \mathcal{N}$ is used to ensure $F$ is orthogonal to the Maass space.  These conditions also relate to the assumptions on the Hecke character $\chi$ (Definition 7.11 of \cite{klosin2009congruences}).  If $\chi$ is chosen to satisfy
$$
\ord_\varpi \left(\prod^2_{j=1}L^{int}(BC(f), j+(t+k)/2,\overline{\chi \omega}) \right) = 0,
$$
then the condition in Theorem \ref{thm_4dim} becomes $\ord_\varpi (L^{int}({\Symm^2 f},k)) >0$.  The existence of a character $\chi$ satisfying this is not known in general (as mentioned in \cite[Remark 7.15]{klosin2009congruences}).

The congruence involving the Maass lift is used to provide evidence for the Bloch-Kato conjecture for special $L$-values of $L^{int}({\Symm^2 f},k)$ by Klosin (see \cite[Sec. 9]{klosin2009congruences}).  Hypothesis (2) of Theorem \ref{thm_4dim} (and hypothesis (1) of Theorem \ref{main_thm}) relate to $L$-values and cases of the Bloch-Kato conjecture, which are generally known in only some cases (such as $n=2$ in Section \ref{sec:2dim}). So, the modularity result above holds under assumptions consistent with the Bloch-Kato conjecture in this case.
\end{myrem}

\section*{Acknowledgements}
I would like to thank my doctoral advisor Krzysztof Klosin for many helpful comments and advice.  I would like to thank the referee for many helpful comments.

\end{document}